\documentclass[11pt]{amsart}
\usepackage{amsmath}
\usepackage{amsfonts}
\usepackage{amsthm}
\usepackage{amssymb}
\usepackage{verbatim}
\usepackage{enumerate}
\usepackage{graphicx}
\usepackage{longtable}
\usepackage[all,cmtip]{xy}
\usepackage{upref}
\usepackage{hyperref}
\usepackage[hang,small,bf]{caption}

\numberwithin{equation}{section}

\newtheorem{definition}{Definition}[section]
\newtheorem{theorem}[definition]{Theorem}
\newtheorem{proposition}[definition]{Proposition}
\newtheorem{remark}[definition]{Remark}
\newtheorem{lemma}[definition]{Lemma}
\newtheorem{corollary}[definition]{Corollary}

\renewcommand{\theta}{\vartheta}

\newcommand{\T}{{\mathbb T}}
\newcommand{\Z}{{\mathbb Z}}
\newcommand{\R}{{\mathbb R}}
\newcommand{\N}{{\mathbb N}}

\newcommand{\op}[1]{\operatorname{#1}}
\renewcommand{\phi}{\varphi}

\begin{document}

\title[The Lusternik-Fet theorem on twisted cotangent bundles]{The Lusternik-Fet theorem \\ for autonomous Tonelli Hamiltonian systems\\ on twisted cotangent bundles}

\author{Luca Asselle}
\address{Ruhr-Universit\"at Bochum, Fakult\"at f\"ur Mathematik, NA 4/35, Universit\"atsstra\ss e 150, D-44780 Bochum, Germany}
\email{\href{mailto:luca.asselle@ruhr-uni-bochum.de}{luca.asselle@ruhr-uni-bochum.de}}
\author{Gabriele Benedetti}
\address{WWU M\"unster, Mathematisches Institut, Einsteinstrasse 62, D-48149 M\"unster, Germany}
\email{\href{mailto:benedett@uni-muenster.de}{benedett@uni-muenster.de}}

\subjclass[2010]{37J45, 58E05}

\keywords{Dynamical systems, Periodic orbits, Symplectic geometry, Magnetic flows}
\date{\today}
\begin{abstract}
Let $M$ be a closed manifold and consider the Hamiltonian flow associated to an autonomous Tonelli Hamiltonian $H:T^*M\rightarrow \R$ and a twisted symplectic form. In this paper we study the existence of contractible periodic orbits for such a flow. Our main result asserts that if $M$ is not aspherical then contractible periodic orbits exist for almost all energies above the maximum critical value of $H$.
\end{abstract}

\maketitle


\section{Introduction}

Let $M$ be a closed manifold and let $\sigma$ be a closed two-form on $M$. We refer to $\sigma$ as the \textit{magnetic form}. We consider the standard symplectic form on $T^*M$ twisted by the pull-back of $\sigma$ via the projection map $\pi:T^*M\rightarrow M$:
\begin{equation*}
\omega_\sigma\ :=\ dp\wedge dq\ +\ \pi^*\sigma\,.
\end{equation*}
Let $H:T^*M\rightarrow\R$ be a Tonelli Hamiltonian; namely $H$ is uniformly convex and grows super-linearly in the fibres (see \cite{Abb13}). We denote by $t\mapsto\Phi_t^{H,\sigma}$ the Hamiltonian flow of $H$ with respect to $\omega_\sigma$. It is generated by the vector field $X_{H,\sigma}$ defined by
\begin{equation*}
\imath_{X_{H,\sigma}}\omega_\sigma\ =\ -\,dH\,.
\end{equation*}
These flows are of special physical interest when the Hamiltonian is of \textit{mechanical} type; namely when
\begin{equation}\label{eq:mec}
H(q,p)\ =\ \frac{1}{2}\, |p|_q^2\ +\ V(q)
\end{equation}
where $|\cdot|$ is the (dual) norm induced by a Riemannian metric $g$ on $M$ and $V:M\rightarrow\R$ is a smooth function. Indeed, these dynamical systems model the motion of a particle of unit mass and charge on a Riemannian manifold $(M,g)$ under the effect of the conservative force $-\nabla V$ and of the stationary magnetic field $\sigma$. Among the mechanical Hamiltonians, we distinguish those which are purely \textit{kinetic} (namely, those such that $V=0$). We denote one such Hamiltonian by $H_{\op{kin}}$. In this case we call $\Phi^{H_{\op{kin}},\sigma}$ the \textit{magnetic flow} of the pair $(g,\sigma)$; its trajectories are then called with slight abuse of terminology \textit{magnetic geodesics}.
\medskip

In this paper we are interested in showing the existence of \textit{contractible} periodic orbits of the flow $\Phi^{H,\sigma}$ on a given energy level $H^{-1}(k)$. Here and in the rest of the paper contractibility has to be understood in $T^*M$ (or, equivalently, in $M$ when we refer to the projection of the orbit on the base manifold) and not in $H^{-1}(k)$.
Our main goal will be to prove a generalization to our setting of the classical theorem of Lusternik-Fet asserting that all closed \textit{not aspherical}\footnote{A manifold is called \textit{aspherical} if all its homotopy groups of degree bigger than $1$ are trivial (see \cite{Lue10} for a survey on this class of manifolds).} Riemannian manifolds admit a non-trivial contractible closed geodesic \cite{LF51} (see also \cite{Bir17,Bir27} for the original idea of Birkhoff when $M$ is a sphere). To this purpose we shall define the energy value
\begin{equation}\label{eq:e0}
e_0(H)\ :=\ \max_{q\in M} \min_{p\in T^*_qM}\ H(q,p)\, .
\end{equation}

We can now state the main theorem of this paper. When $M=S^2$, it has concrete applications to the motion of rigid bodies (see \cite[\S 4]{Nov82} and \cite[Theorem 8]{Koz85}).

\begin{theorem}\label{theoremb}
If $M$ is not aspherical, then $H^{-1}(k)$ carries a contractible periodic orbit, for almost every $k>e_0(H)$. 
\end{theorem}

When $\pi_2(M)\neq0$ and the Hamiltonian is of mechanical type, Theorem \ref{theoremb} was proven for the first time by Novikov in \cite[Theorem, page 35]{Nov82} (see also \cite[Theorem 7]{Koz85}). There the author claims the existence for \textit{every} $k>e_0(H)=\max V$, since he did not take into account some decisive compactness problems, which are today only partially solved (cf.\cite{Con06,Mer10,Abb13}). These issues are also ultimately responsible for the fact that we can get existence only on almost every energy level.

Our strategy of proof is to put Novikov's beautiful geometric idea on solid ground by using some techniques from the dynamics of autonomous Lagrangian systems pioneered by Contreras almost ten years ago \cite{Con06}. First, we build a correspondence between the periodic orbits at level $k$ and the zeros of a closed one-form $\eta_k$ on the space of contractible loops in $M$ with arbitrary period. Second, using an argument \`{a} la Birkhoff and Lusternik-Fet, we employ a minimax method in order to construct zeros of $\eta_k$ as limits of sequences of loops $(\gamma_h)$ such that $|\eta_k|_{\gamma_h}\rightarrow 0$.

One of the contributions of this paper is to show that such sequences do have limit points, provided the periods are uniformly bounded and bounded away from zero. Following \cite{Con06} we obtain this bound for \textit{almost every} value of $k$ by an adaptation of the so-called ``Struwe monotonicity argument" \cite{Str90}. It is an open problem to understand if this can be done for \textit{every} $k$.
\medskip

In view of Theorem \ref{theoremb}, it is natural to ask what happens to contractible periodic orbits when the manifold is aspherical. In this case $\pi_2(M)=0$ and, therefore, $\sigma$ is \textit{weakly exact}; namely its lift $\widetilde\sigma$ to the universal cover $\widetilde M$ of $M$ is exact. We define the \textit{Ma\~n\'e critical value of the pair} $(H,\sigma)$ by
\begin{equation}
c(H,\sigma)\ :=\ \inf_{d\widetilde\theta=\widetilde\sigma}\sup_{\widetilde q\in\widetilde M}\ \widetilde H(\widetilde q,-\widetilde\theta_{\widetilde q})\ \in\R\cup\{+\infty\}\,,
\label{cHsigma}
\end{equation}
where $\widetilde H$ is the lift of $H$ to the universal cover. Observe that $c(H,\sigma)$ is finite if and only if $\widetilde \sigma$ has a \textit{bounded primitive}. This means that there exists $\widetilde\theta\in\Omega^1(\widetilde M)$ with
\begin{equation*}
d\widetilde\theta\ =\ \widetilde\sigma\, , \quad \sup_{\widetilde q\in\widetilde M}\ |\widetilde\theta|_{\widetilde q}\ <\ +\infty\, ,
\end{equation*}
where $|\cdot|$ is the (dual) norm induced by the pull-back to $\widetilde M$ of a Riemannian metric on $M$. It is known that if $\sigma$ is non-exact and $\pi_1(M)$ is amenable, then $\widetilde\sigma$ does not admit bounded primitives \cite[Corollary 5.4]{Pat06}, i.e.\ $c(H,\sigma)=+\infty$.

It is immediate to see that $e_0(H)\leq c(H,\sigma)$ and for kinetic Hamiltonians we also have $e_0(H_{\op{kin}})<c(H_{\op{kin}},\sigma)$ if and only if $\sigma\neq0$.

When the magnetic form is weakly exact and the Hamiltonian is of kinetic type, Will Merry proved the following result about the existence of contractible orbits (see \cite{Mer10} and the forthcoming corrigendum \cite{Mer15}). To stick to his notation, we set $c(g,\sigma):=c(H_{\op{kin}},\sigma)$, where $g$ is the Riemannian metric defining $H_{\op{kin}}$.
\begin{theorem}[Merry, 2010]
Let $\sigma$ be weakly exact and $H_{\op{kin}}$ be a kinetic Hamiltonian. Then, for almost every $k\in(0,c(g,\sigma))$, $H^{-1}_{kin}(k)$ carries a contractible closed magnetic geodesic. 
\end{theorem}

Using the techniques developed for Theorem \ref{theoremb} we can give an alternative proof of Merry's result for an arbitrary Tonelli Hamiltonian. When $\sigma$ is exact, respectively when $c(H,\sigma)<+\infty$, the theorem below was already shown by Contreras \cite{Con06}, respectively by Osuna \cite{Osu05}.
\begin{theorem}\label{theorema}
If $\sigma$ is weakly exact and $H$ is a Tonelli Hamiltonian, then for almost every $k\in(e_0(H),c(H,\sigma))$ the level $H^{-1}(k)$ carries a contractible periodic orbit.
\end{theorem}

In general one cannot expect existence of contractible periodic orbits for energies above $c(H,\sigma)$ when $M$ is aspherical. Indeed, if $g$ is a metric on $M$ with non-positive sectional curvature, there are no non-constant contractible closed geodesics on $M$. Hence, when $\sigma=0$, $H^{-1}_{\op{kin}}(k)$ does not carry contractible periodic orbits, for any $k>c(g,0)=0$. However, to the authors' knowledge it is an open question to determine whether a given aspherical manifold has a metric $g$ with no non-constant contractible closed geodesics. Observe, indeed, that there are examples of aspherical manifolds supporting no metric of non-positive sectional curvature \cite{Dav83,Lee95}.

As we have sketched above, the proof of Theorem \ref{theoremb} (and of Theorem \ref{theorema}) is based on variational methods for autonomous Lagrangian systems. In recent years, 
such methods have also been used to prove the existence of infinitely many (not necessarily contractible) periodic orbits on almost all low energy levels when $M$ is a closed surface, $\sigma$ is exact, and $H=H_{\op{kin}}$ (see 
\cite{AMP13,AMMP14}). In a recent paper, the authors extended this latter result to the case in which $\sigma$ is oscillating but not necessarily exact, under the further assumption that the surface $M$ is not $S^2$ (see \cite{AB14,AB15}). Settling the case of $S^2$ represents a challenging open problem.

\medskip

Finally, we analyze the existence of contractible periodic orbits with energy below $e_0(H)$ without assuming any additional condition on $M$ or on $\sigma$. To this purpose we recall that, if $(W,\omega)$ is a symplectic manifold, a set
$U\subset W$ is said to be \textit{displaceable in} $(W,\omega)$ if there exists a compactly supported Hamiltonian diffeomorphism $\phi:W\rightarrow W$, i.e.\ a time-$1$ map of a time-dependent Hamiltonian flow, such that $\phi(U)\cap U=\emptyset$. 
In \cite{Con06} Contreras observed that, when $\sigma=0$, sublevels $\{H\leq k\}$ are displaceable for $k<e_0(H)$. This fact enables him to apply a result of Schlenk \cite[Corollary 3.2]{Sch06} which guarantees the existence of a contractible periodic orbit
for almost every $k\in[\min H,e_0(H))$. Our simple remark leading to the next theorem is that Contreras' observation carries over to the case of arbitrary $\sigma$ and that Schlenk's abstract result still applies.

\begin{theorem}\label{theoremc}
The energy sublevel set $\{H\leq k\}$ is displaceable in $(T^*M, \omega_\sigma)$, for every $k\in[\min H,e_0(H))$. 
\end{theorem}

\begin{corollary}\label{corollaryc}
The energy level set $H^{-1}(k)$ carries a contractible periodic orbit for almost every $k\in [\min H,e_0(H))$.
\end{corollary}

The proof of Theorem \ref{theoremb} and \ref{theorema} and the proof of Corollary \ref{corollaryc} are substantially different. As we previously remarked, the former is tailored to the class of Tonelli Lagrangian systems while the latter pertains to the broader world of symplectic geometry. Building on work of Ta{\u\i}manov \cite{Tai83}, Abbondandolo showed how to use the first approach to give a proof of Corollary \ref{corollaryc} when $\sigma$ is exact \cite{Abb13}. It is still not known whether his idea can be adapted for an arbitrary $\sigma$. Similarly, the second approach was used by Schlenk in \cite[Corollary 3.6]{Sch06} to reprove Theorem \ref{theoremb} and \ref{theorema}, when $\sigma\neq0$, $H=H_{\op{kin}}$ and $k$ lies in the range $(0,d_1(g,\sigma))$, where
\begin{equation*}
d_1(g,\sigma):=\sup\Big\{k\in(0,+\infty)\ \big|\ \{H_{\op{kin}}\leq k\} \mbox{ is stably displaceable in } (T^*M,\omega_\sigma)\Big\}\,.
\end{equation*}

The fact that $d_1(g,\sigma)$ is positive follows from results of Laudenbach-Sikorav \cite{LS94} and Polterovich \cite{Pol95}. Moreover, in the weakly exact case Merry shows in \cite[Theorem 1.1]{Mer11} that $d_1(g,\sigma)\leq c(g,\sigma)$. It is conjectured in \cite{CFP10} that $d_1(g,\sigma)=c(g,\sigma)$. Such equality holds for non-exact forms on tori since in this case $c(g,\sigma)=+\infty$ and there exists a symplectomorphism $(T^*\T^n,\omega_\sigma) \rightarrow (\R^2\times W,\omega_{\op{st}}\oplus \omega)$,
where $(W,\omega)$ is some symplectic manifold, as shown in \cite[Theorem 3.1]{GK99}.
\medskip

A common feature of both methods is that they yield existence results only for almost every level in a certain energy range. However, it is known that if, under the hypotheses of Theorem \ref{theoremb} or Theorem \ref{theorema} or Corollary \ref{corollaryc}, some fixed energy level $H^{-1}(k)$ is \textit{stable} in $(T^*M,\omega_\sigma)$ \cite[page 122]{HZ94}, then there exists a closed contractible orbit with energy $k$ (see \cite[Ch. 4, Theorem 5]{HZ94} and \cite[Corollary 8.4]{Abb13}). This happens for instance whenever 
$M$ is not-aspherical, $\sigma$ is weakly exact and $k>c(H,\sigma)$, as follows at once from the analysis contained in \cite[Section 4]{Mer10}.

In connection with such results, one would like to understand which energy levels are stable. For example, when $M$ is a closed orientable surface, $\sigma$ is symplectic and $H=H_{\op{kin}}$, stability holds on every low energy level. Such statement can be found in \cite{Ben14a} for $M\neq\T^2$ (see also \cite[Lemma 12.6]{FS07} and \cite[Remark 2.2]{Pat09}). When $M=\T^2$ the global angular form associated to any nowhere vanishing section of $T^*\T^2$ is a stabilizing form. Hence, in Corollary \ref{corollariosuperfici} we get a new proof of the existence of a contractible closed magnetic geodesic for every low energy, a result which was already proven by Ginzburg in \cite{Gin87} and, when $M\neq\T^2$, by Schneider in \cite{Schn11,Schn12a,Schn12b}. 

In this setting, one can also give better lower bounds for the number of contractible closed magnetic geodesics. If $M\neq S^2$, this number turns to be infinite for \textit{every} sufficiently low energy: for $M=\T^2$ this was established independently in \cite{LeC06} and in \cite{Hin09}; for surfaces of higher genus a proof was recently given in \cite{GGM14}. On the other hand, for $M=S^2$ every low energy level carries either two or infinitely many closed magnetic geodesics \cite{Ben14b,Ben15}.

If $M$ has dimension higher than $2$ and $\sigma$ is a symplectic form, proving that low energy levels are stable is an open problem (see \cite{CFP10} where a proof of stability is given in the homogeneous case). However, existence of contractible magnetic geodesics for every low energy levels still holds as Usher proves in \cite{Ush09} building upon previous work of Ginzburg and G{\"u}rel \cite{GG09} (see also \cite{Ker99} for multiplicity results, generalizing \cite{Gin87}, when $\sigma$ is a K\"ahler form). 
\medskip

We end up the introduction by giving a summary of the content of this paper.
In Section \ref{sec:free} we introduce the free-period action $1$-form $\eta_k$ associated to the pair $(H,\sigma)$ and we investigate the compactness properties of the vanishing sequences of $\eta_k$ as well as the completeness of the associated negative ``gradient flow''. This leads naturally to studying the behaviour of $\eta_k$ on the set of short loops, which is the content of Section \ref{sec:con}. The general analysis developed in Section \ref{sec:free} and the geometric information obtained in Section \ref{sec:con} allow us to prove Theorem \ref{theoremb} in Section \ref{sec:gen} and Theorem \ref{theorema} in Section \ref{sec:we} via minimax methods. Applications to stable energy levels on surfaces are given along the way. Finally, in Section \ref{sec:dis} we turn to symplectic techniques and prove Theorem \ref{theoremc} and Corollary \ref{corollaryc}.

\subsection*{Acknowledgements}
The authors are deeply grateful to Will Merry for making the draft of his corrigendum available to them prior to publication. They also warmly thank Alberto Abbondandolo, Viktor Ginzburg and Felix Schlenk for valuable discussions and the anonymous referee for precious suggestions, which helped to improve the article. Luca Asselle is partially supported by the DFG grant AB 360/2-1
"Periodic orbits of conservative systems below the Ma\~n\'e critical energy value".
Gabriele Benedetti is supported by the DFG grant SFB 878.


\section{The free-period action 1-form}\label{sec:free}

In this section we give a variational characterization of the periodic orbits on a given energy level $H^{-1}(k)$ in a suitable space of parametrized loops with arbitrary period. We allow also for non-contractible loops at this point of our discussion, since the results contained here naturally apply in this higher degree of generality.

In order for the variational problem to be well-defined, we require $H$ to be \textit{quadratic at infinity}. This means that $H$ can be written as in \eqref{eq:mec} outside some compact set. This is by no means restrictive since in our arguments we will always take $k$ to lie inside a fixed bounded interval of energies and therefore one can easily find a Tonelli Hamiltonian quadratic at infinity which coincides with $H$ on the energy levels of interest. Moreover, this can be done preserving the value $e_0(H)$ (and also the value $c(H,\sigma)$, when $\sigma$ is weakly exact). Thus, in the remainder of the paper we will assume that $H$ is quadratic at infinity and we will denote by $g$ the Riemannian metric determining the kinetic term at infinity.

\subsection{General setting}\label{sec:pre}
Let $\mathcal L:T^*M\rightarrow TM$ be the \textit{Legendre transform} induced by $H$ and $L:TM\rightarrow\R$ the Fenchel dual of $H$. We call $L$ the \textit{Lagrangian function} of the system. Let $E:TM\rightarrow \R$ denote the \textit{energy function} $E:=H\circ \mathcal L^{-1}$ and observe that
\begin{equation*}
e_0(H)\ =\ \max_{q\in M} \ E(q,0)\,.
\end{equation*}
Since $H$ is quadratic at infinity, $L$ and $E$ are also quadratic at infinity. This means that outside a compact set of $TM$ we have
\begin{equation*}
L(q,v)\ =\ \frac{1}{2}|v|^2_q\ -\ V(q)\,,\quad E(q,v)\ =\ \frac{1}{2}|v|^2_q\ +\ V(q)\,,
\end{equation*}
where $|\cdot|$ is the norm induced by $g$ on $TM$. In particular, $E$ satisfies
\begin{equation}
E_0\, |v|_q^2\ -\ E_1 \ \leq \ E(q,v) \, , \quad \forall \ (q,v)\in TM\,,
\label{disuguaglianzeE}
\end{equation}
with $E_0>0$, $E_1\in \R$ are suitable constants. This inequality will play a crucial role in the proof of Theorem \ref{theorem:ps}.
We push the flow $\Phi^{H,\sigma}$ to $TM$ via $\mathcal L$ and define
\begin{equation*}
\Phi^{L,\sigma}\ :=\ \mathcal L\circ \Phi^{H,\sigma}\circ \mathcal L^{-1}\,.
\end{equation*}

In particular, closed orbits of $\Phi^{L,\sigma}$ contained in $E^{-1}(k)$ correspond to closed orbits of $\Phi^{H,\sigma}$ contained in $H^{-1}(k)$.
In the next lemma, we recall that $\Phi^{L,\sigma}$ has an intrinsic definition in terms of a perturbed Euler-Lagrange equation.
\begin{lemma}\label{lem:el}
Let $U$ be an open subset of $M$ such that $\sigma$ admits a primitive $\theta\in\Omega^1(U)$ on $U$. The restriction of the flow $\Phi^{L,\sigma}$ on $TU$ is the standard Euler-Lagrange flows of the Lagrangian $L+\theta:TU\rightarrow\R$. As a corollary, a curve $s\mapsto\gamma(s)\in M$ satisfies in every coordinate chart the perturbed Euler-Lagrange equations
\begin{equation}\label{eq:el}
\frac{d}{ds}\Big(d_vL\,(\gamma,\gamma')\Big)\ =\ d_qL\,(\gamma,\gamma')\ +\ \sigma_{\gamma}(\cdot, \gamma')
\end{equation}
if and only if $(\gamma,\gamma')\subset TM$ is a trajectory of $\Phi^{L,\sigma}$.
\end{lemma}
Thanks to Equation \eqref{eq:el} we will be able to describe periodic orbits of $\Phi^{L,\sigma}$ by a variational principle on some space of loops in $M$ that we now introduce.

Let $\T:=\R/\Z$ and define $\Lambda:= H^1(\T,M)\times (0,+\infty)$. The set $\Lambda$ can be interpreted as the space of absolutely continuous loops with square integrable weak derivative and with arbitrary period. Indeed, every absolutely continuous loop $\gamma:\R/T\Z\rightarrow M$ with $L^2$-weak derivative can be identified with the pair $(x,T)\in H^1(\T,M)\times (0,+\infty)$, where $x(t):=\gamma(Tt)$. 
Conversely, from $(x,T)\in \Lambda$ we obtain a  loop $\gamma:\R/T\Z\rightarrow M$ by $\gamma(s):=x(s/T)$. To ease the notation, we adopt the identification $\gamma=(x,T)$ throughout the paper and, to avoid confusion, 
we denote with a \textit{dot} the derivatives with respect to $t$ and with a \textit{prime} the derivatives with respect to $s$. We define  
\[l(x)\, :=\, \int_0^1 |\dot x(t)|\, dt \, , \quad e(x)\, :=\, \int_0^1 |\dot x(t)|^2\, dt\]
as the \textit{length}, respectively the $L^2$-\textit{energy} of $x$. We also consider the analogous quantities associated to $\gamma$. Clearly, we have $l(x)=l(\gamma)$ and $e(x)=T\,e(\gamma)$. 

As it is shown in \cite{AS09}, $H^1(\T,M)$ has the structure of a \textit{complete Hilbert manifold}. Its tangent space at $x\in H^1(\T,M)$ is the set of all absolutely continuous vector fields $\zeta$ along $x$ whose covariant derivative $\dot\zeta$ with respect to the Levi-Civita connection of $g$ is $L^2$-integrable. We define the inner product $g_{H^1}$ on each tangent space by
\begin{equation*}
g_{H^1}(\zeta_1,\zeta_2)_x\, :=\int_0^1\Big[g_{x}(\zeta_1,\zeta_2)+g_{x}(\dot\zeta_1,\dot\zeta_2)\Big]dt\,,\quad \forall\ \zeta_1, \zeta_2 \in T_xH^1(\T,M)\,.
\end{equation*}
Moreover, a local chart centered at any \textit{smooth} loop $x$ can be constructed as follows. Let $B_\rho^n\subset \R^n$ be the open ball of radius $\rho$ and let $\psi:\T\times B_\rho^n\rightarrow M$ be a \textit{bi-bounded time-dependent chart} for $M$. This is a smooth map such that for all $t\in\T$:
\begin{itemize}
\item $\psi(t,0)=x(t)$;
\item $\psi(t,\cdot):B^n_\rho\longrightarrow M$ is an embedding;
\item $\psi(t,\cdot)$ and its inverse have bounded $C^\infty$-norm.                                                                                                                                                                                                                                                                                                                                                                                                                                                                                                                   \end{itemize}
To ease the notation, we write 
\[\psi_t\,:=\ \psi(t,\cdot)\, , \quad \dot\psi_t\,:= \ \frac{\partial}{\partial t}\psi(t,\cdot)\, .\]
Notice that, for each $t\in\T$, $\dot\psi_t$ is a smooth section of $\psi_t^*(TM)$. Consider the maps 
\[\iota:H^1(\T,B_\rho^n)\longrightarrow H^1(\T,\T\times B_\rho^n)\, , \quad \hat \Psi_{H^1}:H^1(\T,\T\times B_\rho^n)\longrightarrow H^1(\T,M)\]
defined by
\begin{equation*}
\iota(\xi)(t)\,:=\ (t,\xi(t))\,,\quad\quad \hat\Psi_{H^1}(\hat\xi)(t)\,:=\ \psi(\hat\xi(t))\,.
\end{equation*}
The required local chart around $x$ in $H^1(\T,M)$ is 
\[\Psi_{H^1}\,:=\ \hat \Psi_{H^1}\circ\iota:H^1(\T,B_\rho^n)\longrightarrow H^1(\T,M)\, , \quad \Psi_{H^1}(\xi)(t)\,:=\ \psi_t(\xi(t))\,.\]

The set $\Lambda$ has a natural Hilbert manifold structure given by the product between that of $H^1(\T,M)$ and the standard one on $(0,+\infty)$. This means that the tangent space of $\Lambda$ admits the splitting 
\begin{equation}\label{eq:split}
T\Lambda\ =\ T H^1(\T,M)\ \oplus \ \R \, \frac{\partial}{\partial T}\, ,
\end{equation}
and it carries the product metric
\begin{equation*}
g_{\Lambda}\ :=\ g_{H^1}\ +\ dT^2\,.
\end{equation*}

The distance function $d_{\Lambda}$ induced by $g_{\Lambda}$ is not complete as the second factor is not complete with the Euclidean distance. However, $d_{\Lambda}$ is complete on all subsets of the form 
$H^1(\T,M)\times [T_-,+\infty)$, where $T_->0$ is some positive number. 

We denote the norm on $T\Lambda$ induced by $g_{\Lambda}$ simply by $|\cdot|$.
We also define local charts around $(x,T)\in\Lambda$ by
\[\Psi_{\Lambda}\,:=\ \Psi_{H^1}\times \op{Id}_{(0,+\infty)}:H^1(\T,B_\rho^n)\times (0,+\infty)\longrightarrow\Lambda\]

\noindent and observe that the connected components of $\Lambda$ correspond to the free homotopy classes $[\T,M]$ of loops in $M$. Throughout the whole paper we will denote by $\Lambda_0$ the connected component of contractible loops. 
\medskip

We are now ready to introduce the action 1-form $\eta_k\in \Omega^1(\Lambda)$. First, we consider the free-period Lagrangian action functional associated with $L$
\begin{equation}
S^L_k:\Lambda\longrightarrow\R\, , \quad S^L_k(x,T):=\ T\cdot\int_0^1 \Big[L\Big (x(t),\frac{\dot x(t)}{T}\Big )+k\Big]\, dt\,.
\end{equation}
The functional $S^L_k$ is well-defined, since $L$ is quadratic at infinity. We set
\begin{equation}\label{etak}
\eta_k\ :=\ dS^L_k\ +\ \pi_{H^1}^*\tau^\sigma\,,
\end{equation}
where $\pi_{H^1}:\Lambda\rightarrow H^1(\T,M)$ is the standard projection and 
$\tau^\sigma\in\Omega^1(H^1(\T,M))$ is the \textit{transgression of} $\sigma$:\begin{equation}
\tau^\sigma_x[\zeta]\ :=\ \int_0^1\sigma_{x(t)}(\zeta(t),\dot x(t))\,dt\,,\quad \forall\,\zeta\in T_xH^1(\T,M)\,.
\label{tauk}
\end{equation}

If $\gamma\in\Lambda$ we can compute how $\eta_k$ acts on $T_\gamma\Lambda$ using the splitting \eqref{eq:split}. If $\zeta\in T_xH^1(\T,M)$ and $\alpha(s):=\zeta(s/T)$, we find
\begin{equation}\label{eq:dif1}
(\eta_k)_{\gamma}\big[(\zeta,0)\big]\ =\ \int_0^T\Big[\,d_qL(\gamma,\gamma')[\alpha]\,+\, d_vL(\gamma,\gamma')[\alpha']\,+\, \sigma_\gamma(\alpha,\gamma')\,\Big]\, ds\,.
\end{equation}
In the direction of the period we find
\begin{align}\label{par-T}
(\eta_k)_\gamma\left[\frac{\partial}{\partial T}\right]\ =\ dS^L_k\left[\frac{\partial}{\partial T}\right]\ &=\ k\ -\ \int_0^1 E\Big(x(t),\frac{\dot x(t)}{T}\Big )\, dt\nonumber\\
&=\ k\ -\ \frac{1}{T}\int_0^T E\big(\gamma(s),\gamma'(s)\big)\, ds.
\end{align}
Arguing as in \cite[Lemma 2.1]{Con06} and making use of Equations \eqref{eq:dif1} and \eqref{par-T} we arrive at the following characterization of the zeros of $\eta_k$. 
\begin{lemma}\label{zeridietak}
A loop $\gamma\in\Lambda$ satisfies $(\eta_k)_\gamma=0$ if and only if $(\gamma,\gamma')\subset TM$ is a periodic orbit of $\Phi^{L,\sigma}$ contained in $E^{-1}(k)$. 
\end{lemma}

We write now $\eta_k$ in a local chart in order to investigate its regularity as a section of the bundle $T^*\Lambda\rightarrow\Lambda$. Our computations follow Section 3 in \cite{AS09}. We decided to include them here anyway since in the original paper, there is a little inaccuracy in the definition of the pull-back Lagrangian to a local chart.

\begin{lemma}\label{lem:lagloc}
Let $\psi$ be a bi-bounded time-dependent chart for $M$ and let $\Psi_{\Lambda}$ be the associated local chart of $\Lambda$. There exists a smooth function 
\[L_{\psi}:\T\times TB_\rho^n\times(0,+\infty)\longrightarrow\R\]
such that $\Psi_{\Lambda}^*\eta_k=dS^{L_{\psi}}_k$, where $S^{L_{\psi}}_k:H^1(\T,B^n_\rho)\times(0,+\infty)\rightarrow\R$ is defined by
\begin{equation}\label{eq:lagloc}
S^{L_{\psi}}_k(y,T):=\ T\cdot\int_0^1 \Big[L_{\psi}\Big (t,y(t),\frac{\dot y(t)}{T},T\Big )+k\Big]\, dt\,.
\end{equation}

For any $T_->0$ and for any $T\geq T_-$ the family of functions $L_{\psi}(\,\cdot\,,T):\T\times TB_\rho^n\rightarrow\R$ is Tonelli and quadratic at infinity uniformly in $T$. Namely, there exist positive constants $L_0,L_1,L_2$ depending on $T_-$ such that for all $(t,y,\xi)\in \T\times TB_\rho^n$
\begin{align}
d_{\xi\xi}L_\psi(t,y,\xi,T)\ &\geq  \ L_2\,;\label{tonloc1}\\
\lim_{|\xi|\rightarrow+\infty}\frac{L_\psi(t,y,\xi,T)}{|\zeta|}\ &=\ +\infty\,,\quad\quad\mbox{uniformly in }T\in [T_-,+\infty)\,;\label{tonloc2}\\
|d_y L_\psi(t,y,\xi,T)|\ &\leq\ L_0(1+|\xi|^2)\,,\quad |d_\xi L_\psi(t,y,\xi,T)|\ \leq\ L_1(1+|\xi|)\,.\label{qualoc}
\end{align}
\end{lemma}
\begin{proof}
First, we express the velocity of $\gamma=(x,T)$ in a local chart. Let $\beta:\R/T\Z\rightarrow B_\rho^n$ be given by $\gamma(s)=\psi_{s/T}(\beta(s))$ and compute
\begin{align*}
\gamma'(s ) \ = \ \frac{d}{ds}\left [\psi_{s/T}\left(\beta(s)\right)\right ](s) \ &=\  \frac{\dot\psi_{s/T}(\beta(s))}{T}\ +\ d_{\beta(s)}\psi_{s/T}\big [\beta'(s)\big ]\\
                    &=\ \frac{\dot\psi_{t}(y(t))}{T}\ +\ d_{y(t)}\psi_{t}\Big [\frac{\dot y(t)}{T}\Big ]\,,
\end{align*}

\noindent where $t=s/T\in\T$ and $y(t):=\beta(Tt)$. If we define $D\psi:\T\times TB_\rho^n\times(0,+\infty)\rightarrow TM$ by
\begin{equation*}
D\psi(t,y,\xi,T)\,:=\ \left(\psi_t(y)\, ,\ \frac{\dot\psi_t(y)}{T}\ +\ d_{y}\psi_t[\xi]\right)\,,
\end{equation*}
then we have
\begin{equation}\label{eq:locder}
\left(x\, ,\ \frac{\dot x}{T}\right)\ =\ D\psi\left(t\, ,\ y\, ,\ \frac{\dot y}{T}\, ,\ T\right)\,.
\end{equation}
Using \eqref{eq:locder} we write $S^L_k$ in a local chart. If $(y,T)\in H^1(\T,B^n_\rho)\times(0,+\infty)$, then
\begin{equation*}
(S^L_k\circ \Psi_{\Lambda})(y,T)\ = \ T\cdot \int_0^1 \Big[\,(L\circ D\psi)\Big (t,y,\frac{\dot y}{T},T\Big )+k\,\Big]\, dt\,,
\end{equation*}
Since $H^2(\T\times B_\rho^n,\R)=0$, there is $\hat\theta\in\Omega^1(\T\times B_\rho^n)$ with $\psi^*\sigma=d\hat\theta$ and hence we get 
\[\hat\Psi_{H^1}^*\tau^\sigma\ =\ \tau^{\psi^*\sigma}\ =\ dS^{\hat\theta}\, ,\]
where 
\begin{equation*}
S^{\hat\theta}:H^1(\T,\T\times B_\rho^n)\longrightarrow \R\, , \quad S^{\hat\theta}(\hat y)\,:=\,\int_0^1\hat y^*\hat\theta\,.
\end{equation*}
For every $(t,y)\in\T\times B_\rho^n$, we can write 
\[\hat\theta_{(t,y)}\ =\ V_t(y)dt\ +\ \theta^t_y\]
for suitable $V_t:B_\rho^n\rightarrow \R$ and $\theta^t\in\Omega^1(B_\rho^n)$. By definition of $\iota$ and $\Psi_{H^1}$ we get 
\begin{equation*}
\Psi_{H^1}^*\tau^\sigma\ =\ d(S^{\hat\theta}\circ\iota)\quad \mbox{and}\quad S^{\hat\theta}(\iota(y))\ =\ \int_0^1 \Big[V_t(y(t))\,+\,\theta^t_{y}(\dot y)\Big]\,dt\,.
\end{equation*}
Thus, \eqref{eq:lagloc} holds if we set
\begin{equation}
L_{\psi}(t,y,\xi,T):=\ (L\circ D\psi)(t,y,\xi,T)\ +\ \frac{V_t(y)}{T}\ +\ \theta^t_y(\xi)\,.
\end{equation}
Since $L$ is Tonelli and $\psi$ bi-bounded, relations \eqref{tonloc1}, \eqref{tonloc2} and \eqref{qualoc} follow.
\end{proof}

\begin{corollary}\label{cor:reg}
The $1$-form $\eta_k$ is locally uniformly Lipschitz. Moreover, its integral over any closed differentiable path $u:\T\rightarrow \Lambda$ depends only on the free homotopy class of $u$. In this sense, we say that $\eta_k$ is $\mathsf{closed}$.
\end{corollary}

\begin{proof}
By Lemma \ref{lem:lagloc} and the computations in \cite[Lemma 3.1(i)]{AS09}, $\eta_k$ is locally the differential of a $C^1$-function with locally bounded derivatives. This observation implies both of the statements that we have to prove. 
\end{proof}

In general $\eta_k$ is not globally \textit{exact} on $\Lambda$. To this purpose we observe that $\sigma$ is weakly-exact if and only if $[\sigma]\in H^2(M,\R)$ vanishes over $\pi_2(M)$. As we discuss in Section \ref{sec:we}, this means that $\eta_k$ is \textit{exact} on $\Lambda_0$ if and only if $\sigma$ is weakly-exact. However, even if $\sigma$ is weakly exact it can still happen that $\eta_k$ is not globally exact on $\Lambda$. For example, if $M=\T^2$ and $[\sigma]\neq0\in H^2(\T^2,\R)$, then $\eta_k$ is not exact on any connected component of $\Lambda\setminus \Lambda_0$. On the other hand, if the lift of $\sigma$ to the universal cover admits a bounded primitive, then $\eta_k$ is exact on $\Lambda$ (see \cite[Lemma 2.2]{Mer10}).


\subsection{A compactness criterion for vanishing sequences}

In view of Lemma \ref{zeridietak}, in order to find periodic orbits for $\Phi^{L,\sigma}$ we have to show that the set of zeros of $\eta_k$ is non-empty. The mechanism we use to construct such zeros is to look for limit points of vanishing sequences for $\eta_k$. These sequences are the generalization of Palais-Smale sequences to our setting.

\begin{definition}
We call $(\gamma_h)\subset\Lambda$ a $\mathsf{vanishing \ sequence}$ for $\eta_k$, if 
\[|\eta_k|_{\gamma_h}\ \longrightarrow\ 0\, .\]
\end{definition}

Since $\eta_k$ is continuous, the set of limit points of vanishing sequences coincides with the set of zeros of $\eta_k$. Therefore, it is crucial to know under which hypotheses a vanishing sequence has a limit point. Clearly, if $T_h\rightarrow 0$ or $T_h\rightarrow \infty$, the limit points set is empty. The following theorem shows that the converse is also true.

\begin{theorem}\label{theorem:ps}
Let $(\gamma_h)_{h\in\N}$ be a vanishing sequence for $\eta_k$ in a given connected component of $\Lambda$. Then there exists $C>0$ such that
\begin{equation}\label{en-per}
e(x_h)\ \leq\ CT_h^2\, , \quad \forall \ h\in \N\, . 
\end{equation}
As a consequence, the following two statements hold:
\begin{enumerate}
\item If $T_h$ tends to zero, then $\ e(x_h)=O(T_h^2)$.
\item If $T_h$ is uniformly bounded and bounded away from zero, then $(\gamma_h)_{h\in\N}$ has a converging subsequence. 
\end{enumerate}
\end{theorem}

\begin{proof}
Since the vector field $\frac{\partial}{\partial T}$ has norm $1$, we have
\begin{equation}
\left|\eta_k\left[\frac{\partial}{\partial T}\right]\right|\ \leq\ \big|\eta_k\big|
\end{equation}
Hence, by \eqref{par-T} and the fact that $(\gamma_h)$ is a vanishing sequence, we get
\[\alpha_h\ := \ \frac{1}{T_h}\int_0^{T_h} \Big [ E\Big (\gamma_h(s),\gamma'_h (s) \Big ) - k\Big ]\, ds \ =\ -\, (\eta_k)_{\gamma_h}\left[\frac{\partial}{\partial T}\right] \ \longrightarrow \ 0 \, .\]
On the other hand, since $E$ is quadratic at infinity, \eqref{disuguaglianzeE} yields
\[\alpha_h \ \geq \ \frac{E_0}{T_h}\, e(\gamma_h) \  - \ (E_1+k)\]
Therefore,
\[e(x_h)\ =\ T_he(\gamma_h) \ \leq \ \ \frac{T^2_h}{E_0} \, \big(\alpha_h\, +\, k\,+\,E_1\big)\]
hence proving \eqref{en-per}. Statement (1) follows at once from what we have just proved.

We now show (2). Since the periods are by assumption uniformly bounded from above, we know by \eqref{en-per} that $e(x_h)$ is also uniformly bounded. Hence, the curves $x_h$ are uniformly $1/2$-H\"older continuous. By the Ascoli-Arzel\`a theorem, up to a subsequence they converge uniformly to a continuous curve $x$. Thus, there exists a smooth curve $x_0\in H^1(\T,M)$ such that $x$ and, up to a subsequence, all the $x_h$ belong to the image of $\psi_{x_0}$. Let us define $y$ and $y_h$ as $x=\Psi_{H^1}(y)$ and $x_h=\Psi_{H^1}(y_h)$. Taking $T=1$ in \eqref{eq:locder} and using the fact that $\psi^*_tg$ is equivalent to the Euclidean metric on $\R^n$, we see that $y_h$ is a bounded sequence in $H^1(\T,B^n_\rho)$. Hence $y\in H^1(\T,B^n_\rho)$ and, up to a subsequence, $y_h$ converges $H^1$-weakly to $y$. Exploiting once again that $(\gamma_h)$ is a vanishing sequence, we have
\begin{equation}
o(1)\ =\ d_{(y_h,T_h)}S^{L_\psi}_k\big[(y_h-y,0)\big]\,.
\end{equation}

Now one argues as in the part of the proof of Lemma 5.3 in \cite{Abb13} after Equation $(5.2)$ and the thesis follows.
\end{proof}

Our strategy to prove Theorem \ref{theoremb} and \ref{theorema} is to construct vanishing sequences via a minimax method. For the argument we will make use of a vector field on $\Lambda$ which generalizes the gradient of $S_k^{L}$. In the next subsection, we briefly discuss what are the properties of this vector field. 


\subsection{The action variation along a path and the gradient of the action}

We know that when the $1$-form $\tau^\sigma$ is non-exact, a global primitive $S_k$ of $\eta_k$ on $\Lambda$ does not exist. However, if $u:[0,1]\rightarrow \Lambda$ is of class $C^1$, the variation
$\Delta S_k(u)\in\R$ of $\eta_k$ along the path $u$ is always well defined. It is given by the formula
\begin{equation}
\Delta S_k(u)(a)\, :=\ \int_0^a u^*\eta_k\,.  
\label{variazionelungou}
\end{equation}

Then, since $\eta_k$ is closed, we extend the definition of $\Delta S_k$ to any continuous path by uniform approximation with paths of class $C^1$.
Observe that if $z:Z\rightarrow\Lambda$ is a smooth map such that $z^*\eta_k$ admits a primitive $S_k(z):Z\rightarrow\R$, then for every path $u:[0,1]\rightarrow Z$, there holds
\begin{equation}\label{eq:var}
\Delta S_k(z\circ u)(a)\ =\ S_k(z)(u(a))\ -\ S_k(z)(u(0))\,,\quad\forall\,a\in[0,1]\,.
\end{equation}

The next lemma describes how $\Delta S_k$ changes under deformation of paths with the first endpoint fixed. The proof follows from the fact that $\eta_k$ is a closed form.

\begin{lemma}\label{lem:dec}
Let $R>0$ and suppose that $u:[0,R]\times[0,1]\rightarrow \Lambda$ is a homotopy of paths. Denote by $u_r:=u(r,\cdot)$ and $u^a:=u(\cdot,a)$ the paths in $\Lambda$ obtained keeping one of the variables fixed. If $u^0$ is constant, then
\begin{equation}\label{eq:dec}
\Delta S_k(u_R)(1)\ =\ \Delta S_k(u_0)(1)\ +\ \Delta S_k(u^1)(R)\,.
\end{equation}
\end{lemma}

We now proceed to define the desired gradient. First we consider the vector field $-\sharp\eta_k$, where $\sharp$ is the duality between $T\Lambda$ and $T^*\Lambda$ given by the metric $g_{\Lambda}$. By Corollary \ref{cor:reg} $-\sharp\eta_k$ is locally uniformly Lipschitz and hence we have local existence and uniqueness for solutions of the associated Cauchy problem for positive times. However, the maximal solutions of the Cauchy problem might not be defined for \textit{all} positive times for two reasons. First, $\Lambda$ is not metrically complete so that the period could go to zero in finite time. Second, $|\eta_k|$ is not uniformly bounded on $\Lambda$ so that the trajectories of the flow could escape to infinity in finite time. To avoid the latter problem we consider the normalized vector field 
\begin{equation}
X_k\ :=\ \frac{-\,\sharp\eta_k}{\sqrt{1+|\eta_k|^2}}\,.
\end{equation}
We define $\Phi^{k}$ to be the positive semi-flow of $X_k$ on $\Lambda$. Its flow lines are then the maximal solutions 
\[u_{(x,T)}:\ [0,R_{(x,T)})\ \longrightarrow\ \Lambda\]
of the Cauchy problem associated to $X_k$ with initial condition $u_{(x,T)}(0)=(x,T)$. Here $(x,T)$ is any element in $\Lambda$ and $R_{(x,T)}$ is some number in $(0,+\infty]$. By definition, we have 
\[\Phi^{k}_r(x,T)\ =\ u_{(x,T)}(r)\ =: (x(r),T(r))\, .\]
We say that $\Phi^k$ is \textit{complete}, if $R_{(x,T)}=+\infty$ for all $(x,T)$. It is shown in \cite[Example 6.8]{Con06} that the semi-flow $\Phi^k$ is not complete on $\Lambda_0$, if $k>\min E$. However, the next result poses some restrictions on the flow lines with finite maximal interval of definition. Our proof follows \cite[Lemma 6.9]{Con06}.

\begin{proposition}\label{pro:comp}
Let $u:[0,R)\rightarrow \Lambda$ be a maximal flow line of $\Phi^k$. If $R<+\infty$, then there exists a sequence $r_h\rightarrow R$ such that
\begin{equation}\label{eq:infper}
T(r_h)\ \longrightarrow\ 0\quad \text{and} \quad e(x(r_h))\ = \ O(T(r_h)^2)\, .
\end{equation}
\end{proposition}

\begin{proof}
First, arguing by contradiction we prove that
\begin{equation}\label{eq:liminf}
\liminf_{r\rightarrow R}\,T(r)\ =\ 0\,.
\end{equation}
Thus, we assume that $T(r)\geq T_*>0$ for every $r\in[0,R)$ and observe that
\begin{equation*}
\left|\frac{du}{dr}\right|\ =\ \left|\frac{-\sharp\eta_k}{\sqrt{1+|\eta_k|^2}}\right|\ =\ \frac{|\eta_k|}{\sqrt{1+|\eta_k|^2}}\ <\ 1\,.
\end{equation*}
Since $H^1(\T,M)\times[T_*,+\infty)$ is complete and the derivative of $u$ is bounded by the above inequality, there exists the limit 
\[u_*\ := \ \lim_{r\rightarrow R}u(r)\, .\]

\noindent As $X_k$ is locally uniformly Lipschitz, there exists a neighbourhood $\mathcal U$ of $u_*$, such that the solutions to the Cauchy problem with initial data in $\mathcal U$ all exist in a small fixed interval $[0,r_{u_*}]$. This yields a contradiction as soon as $u(r)\in\mathcal U$ and $R-r<r_{u_*}$. Suppose now that \eqref{eq:liminf} holds. In this case there is a sequence $r_h\rightarrow R$ such that 
\[T(r_h)\ \longrightarrow\ 0 \quad \text{and} \quad \frac{dT}{dr}(r_h)\ \leq \ 0\, .\]

\noindent Using \eqref{disuguaglianzeE}, we find
\begin{align*}
0\ \geq\ \frac{dT}{dr}(r_h) \ =\  -\,(\eta_k)_{u(r_h)}\left[\frac{\partial }{\partial T}\right]\ &=\ \int_0^1 E\Big(x(r_h),\frac{\dot x(r_h)}{T}\Big)\, dt\ -\ k \\ 
                                             &\geq \ E_0\, \frac {e(x(r_h))}{T(r_h)^2}\ -\ E_1\ -\ k\,,
\end{align*}

\noindent which gives the required bound for the energy. 
\end{proof}

The previous proposition shows that the only source of non-completeness of the semi-flow are trajectories that go closer and closer to the subset of constant loops. We have encountered a similar phenomenon in Theorem \ref{theorem:ps}(1), where we saw that a possible class of vanishing sequences without limit points are those with infinitesimal length. It is therefore necessary to take a better look to the behaviour of $\eta_k$ on short loops; we will do this in the next section. 

Before moving on to this task, we end this subsection by generalizing to our setting one of the standard estimates for gradient flows (see for example \cite[Lemma 6.9]{Con06}), namely the bound of the $1/2$-H\"older norm of a flow-line in terms of the action variation and the length of the interval. This result will be used in the proof of Proposition \ref{Struwe} in order to show that the period bounds obtained by the Struwe's monotonicity argument are preserved by the semi-flow. 

\begin{lemma}\label{lem:ac-distper}
If $u:[0,R]\rightarrow\Lambda$ is a flow line of $\Phi^{k}$, then
\begin{equation}
d_\Lambda(u(R),u(0))^2\ \leq\ R\cdot\big(-\Delta S_k(u)(R)\big)\,.
\end{equation}
In particular,
\begin{equation}\label{eq:ac-per}
|T(R)-T(0)|^2\ \leq \ R\cdot\big(-\Delta S_k(u)(R)\big)\,.
\end{equation}
\end{lemma}

\begin{proof}
We start from the definition of the action variation and compute
\begin{align*}
R\cdot\big(-\Delta S_k(u)(R)\big)\ &=\ R\, \int_0^R\!\!-\,\eta_k\left[\frac{du}{dr}\right]dr\ \geq\ R\, \left(\int_0^R\left|\frac{du}{dr}\right|^2\!\!dr\right) \\ 
                            &\geq\ \left(\int_0^R\left|\frac{du}{dr}\right|dr\right)^2\ \geq\ d_{\Lambda}(u(R),u(0))^2\,, 
\end{align*}

\noindent where in the last but one inequality we used the Cauchy-Schwarz inequality. To obtain \eqref{eq:ac-per} we just observe that 
\[|T(R)-T(0)|\ \leq \ d_{\Lambda}(u(R),u(0))\]
as $d_{\Lambda}$ is a product distance. 
\end{proof}


\section{The geometry of the action 1-form on the subset of short loops}\label{sec:con}

We now study the action form on the subset of short loops. The main goal of this section will be to show that $\eta_k$ admits a primitive in a neighbourhood of the constant loops and that such a primitive has an interesting geometry when $k>e_0(H)$.

Thus, let 
\[M_+\ :=\ M\times (0,+\infty)\ \subset\ \Lambda_0\quad\mbox{and}\quad M_T\ :=\ M\times\{T\}\ \subset\ M_+\]
denote the set of constant loops with arbitrary period and the set of constant loops with fixed period $T$, respectively. For every $\delta>0$ consider the neighbourhood $\mathcal V_\delta$ of $M_+$ in $\Lambda_0$ given by 
the loops with length less than $\delta$
\begin{equation}
\mathcal V_\delta\ :=\ \Big\{(x,T)\in \Lambda_0 \ \Big |\ l(x)<\delta\Big\}\,.
\end{equation}
It is well-known that, if $\delta$ is sufficiently small, then there is a deformation retraction of $\mathcal V_\delta$ onto $M_+$ which fixes the period. This yields a map $D_x:B^2\rightarrow \mathcal V_\delta$ such that $D_x|_{\partial B^2}=x$. We use the capping disc $D_x$ to define a primitive $S_k$ for $\eta_k$ on $\mathcal V_\delta$:
\begin{equation}
S_k(x,T)\,  :=\ S^L_k(x,T)\ +\ \int_{B^2}D_x^*\,\sigma\,.
\label{primitiveonvdelta}
\end{equation}
The integral of $\sigma$ over the capping disc can be easily estimated. Indeed, Lemma 7.1 in \cite{Abb13} states that, if $\delta >0$ is sufficiently small, then there exists a constant $\Theta_0>0$ such that, for every $(x,T)\in\mathcal V_\delta$,
\begin{equation}\label{eq:area}
\left|\int_{B^2}D_x^*\,\sigma\right|\ \leq\ \Theta_0\, l(x)^2\, .
\end{equation}

For the proof of the main theorem, we will need bounds on the local primitive $S_k$ which are uniform for $k$ lying in a compact range of energies. Therefore, we assume for the rest of this section that $k$ belongs to $I$, a fixed open interval with compact closure in $(e_0(H),+\infty)$. 

We start by estimating the function $S_k$ from above.

\begin{lemma}
If $\delta >0$ is small enough, then there exists $B>0$ such that, $\forall\,k\in I$,
\begin{equation}\label{eq:estab}
S_k(x,T)\ \leq\ B\cdot\left(\frac{e(x)}{T}\, +\, T\, +\, l(x)^2\right)\,, \ \ \ \ \forall \ (x,T) \in \mathcal V_\delta\, .
\end{equation}
\end{lemma}

\begin{proof}
Since $L$ is quadratic at infinity we can find $B_0>0$ such that 
\[L(q,v)\ \leq\ B_0(1+|v|^2_q)\, , \ \ \ \ \forall \ (q,v)\in TM\]

\noindent and we readily compute using \eqref{eq:area}
\begin{align*}
S_k(x,T)\ &=\ T\int_0^1\Big[L\Big (x(t),\frac{\dot x(t)}{T}\Big )+k\Big]\, dt\ +\ \int_{B^2}D_x^*\,\sigma\\
&\leq\ T\int_0^1\Big[B_0\, \frac{|\dot x(t)|^2}{T^2}+B_0+k\Big]\, dt\ +\ \Theta_0\,l(x)^2\\
&\leq\  B_0\, \frac{e(x)}{T}\, +\, (B_0+\sup I)T\, +\, \Theta_0\, l(x)^2\,.
\end{align*}
\end{proof}

\noindent The value of $S_k$ at $(x,T)\in M_{T}$ can be computed exactly:
\begin{equation}\label{eq:skonm}
S_k(x,T)\ =\ T\big [L(x,0)+k\big ]\ =\ T\big [k-E(x,0)\big ]\,.
\end{equation}

In particular, we can choose $T$ small enough so that the function $S_k|_{M_T}$ is arbitrarily close to zero, uniformly in $k\in I$. On the other hand, the next lemma shows that $S_k$ is positive on $\mathcal V_\delta$ and has a positive lower bound on $\partial \mathcal V_\delta$. Here the hypothesis that $I$ has compact closure in $(e_0(H),+\infty)$ plays a crucial role. 

\begin{lemma}\label{lem:low}
There exist $\delta=\delta(I)>0$ and $\varepsilon=\varepsilon(I)>0$ such that for all $k\in I$ 
\begin{equation}\label{eq:mp}
\inf_{\partial \mathcal V_{\delta}}S_k\ \geq\ \varepsilon\, , \quad\mbox{and} \ \ \ \inf_{\mathcal V_{\delta}}S_k\ =\ 0\,. 
\end{equation}
\end{lemma}

\begin{proof}
The last inequality in the proof of Lemma 7.2 in \cite{Abb13} implies that, for $\delta$ sufficiently small, every $(x,T)\in\mathcal V_\delta$ satisfies
\begin{equation*}
S^L_k(x,T)\ \geq\ l(x)\sqrt{L_2(k-e_0(H))}\ -\ \Theta_1l(x)^2\,,
\end{equation*}

\noindent for suitable $L_2,\Theta_1>0$. Putting together this estimate with \eqref{eq:area},we get
\begin{equation*}
S_k(x,T)\ \geq\ l(x)\sqrt{L_2(k-e_0(H))}\ -\ (\Theta_0+\Theta_1)l(x)^2,
\end{equation*}

\noindent which is positive and bounded away from zero if $l(x)$ is small and positive and $k\in I$. Combining this information with \eqref{eq:skonm}, the thesis follows.
\end{proof}

\noindent In view of this result we define for every $k\in I$ the sublevel sets 
\[\mathcal W'_k\ :=\ \Big \{S_k<\frac{\varepsilon}{4}\Big \}\,\cap\,\mathcal V_{\delta} , \ \  \ \ \mathcal W_k\ :=\ \Big \{S_k<\frac{\varepsilon}{2}\Big \}\,\cap\,\mathcal V_{\delta} .\]
The relevance of these sets to our discussion is demonstrated in the next proposition.
\begin{proposition}\label{lem:wdelta}
There exists a positive real number $T=T(I)$ such that
\begin{equation}\label{eq:wd}
M_{T}\ \subset\ \mathcal W'_{k}\ \subset\ \overline{\mathcal W}_{k}\ \subset\ \mathcal V_{\delta}\,,\quad\forall\, k\in I\,.
\end{equation}
Moreover, the following two statements hold for every $k\in I$:
\begin{enumerate}
 \item If $(\gamma_h)\subset\Lambda_0$ is a vanishing sequence for $\eta_k$ such that $T_h\rightarrow 0$, then $\gamma_h\in \mathcal W'_{k}$, for $h$ large enough.
 \item If a flow line of $\Phi^k$ is not defined for all positive times, then it enters $\mathcal W'_{k}$.
\end{enumerate}
\end{proposition}
\begin{proof}
The inclusions in \eqref{eq:wd} follow at once from Lemma \ref{lem:low} and Equation \eqref{eq:skonm}. We now prove (1). By \eqref{en-per} there is $C>0$ such that $e(x_h)\leq C T_h^2$. Therefore, $(x_h,T_h)\in \mathcal V_{\delta}$ for $h$ big enough and we can use \eqref{eq:estab} to obtain
\begin{equation}
S_k(x_h,T_h)\ \leq\ B\cdot\left(\frac{CT_h^2}{T_h}\, +\, T_h\, +\, CT_h^2\right)\,,
\end{equation}
which goes to zero as $h$ goes to infinity.

We now prove (2). Let $(x,T)\in\Lambda_0\setminus\mathcal W'_{k}$ and let $[0,R_{(x,T)})$ be the maximal interval of definition of the flow line 
\[r\mapsto (x(r),T(r))\ = \ \Phi_r^k(x,T)\, .\]
Suppose that $R_{(x,T)}<+\infty$. By Proposition \ref{pro:comp}, there exists $r_h\rightarrow R$ such that 
\[T(r_h)\ \longrightarrow\ 0\ \ \ \ \text{and} \ \ \ \ e(x(r_h))\ \leq\ C T(r_h)^2\, .\]
for some $C>0$. As in (1) we conclude that $(x(r_h),T(r_h))\in \mathcal W'_{k}$ for large $h$.
\end{proof}

In view of this proposition, we can make the semi-flow $\Phi^{k}$ complete for every $k\in I$ by stopping the flow lines entering $\mathcal W'_{k}$. 
To this purpose we consider a smooth cut-off function $\kappa:(0,+\infty)\rightarrow[0,1]$ such that 
\[\kappa^{-1}(0)\ =\ \big (0,\varepsilon/4\big ]\, , \quad \,\kappa^{-1}(1)\ =\ \big [ \varepsilon/2,+\infty\big ) .\]
We use $\kappa$ to construct a function $\hat\kappa:\Lambda_0\rightarrow[0,1]$ as follows:
\[\hat\kappa\ = \begin{cases}
\ 1 & \text{on} \ \Lambda_0\setminus\mathcal V_{\delta}\,,\\
\ \kappa\circ S_k & \text{on} \ \mathcal V_{\delta}\,. 
\end{cases}
\]
Finally, we define the vector field $\hat X_{k}:=\hat\kappa\,  X_{k}$ and denote its semi-flow by $\hat \Phi^{k}$. We end this section by summarizing the properties of this semi-flow.

\begin{proposition}\label{lem:com}
The semi-flow $\hat \Phi^{k}$ is complete on $\Lambda_0$ for all $k\in I$. Moreover,
\[\hat \Phi^{k}\ = \begin{cases}
\ \Phi^k & \text{on} \ \Lambda_0\setminus\mathcal W_{k}\,,\\
\ \op{Id} & \text{on} \ \mathcal W'_{k}\,. 
\end{cases}
\]
\end{proposition}


\section{Contractible periodic orbits on non-aspherical manifolds}\label{sec:gen}

In this section we give a proof of Theorem \ref{theoremb}. Thus, let $M$ be a closed non-aspherical manifold and $\sigma$ be a closed 2-form on $M$. As in the previous section, fix $I$ an open interval with compact closure in $(e_0(H),+\infty)$ 
and let $\delta=\delta(I)$, $\varepsilon=\varepsilon(I)$ as in the statement of Lemma \ref{lem:low}. Finally, let $T=T(I)$ be as in Proposition \ref{lem:wdelta}.

First, we are going to see that by Lemma \ref{lem:low} the $1$-form $\eta_k$ exhibits a minimax geometry on some class $\mathfrak U$ of continuous maps 
\[(B^{l-1},S^{l-2})\ \longrightarrow\ (\Lambda_0,M_{T})\,.\]

Then, by a generalization of the \textit{Struwe monotonicity argument}, the minimax geometry will yield the existence of zeros of $\eta_k$ on $\Lambda_0$ for almost all $k\in I$.
 
Let us start by constructing $\mathfrak U$. Since $M$ is not-aspherical, there exists a natural number $l\geq2$ such that $\pi_l(M)\neq0$. It is a well-known fact (see for instance \cite[Proposition 2.1.7]{Kli78}) that there exists a bijection
\begin{equation*}
\Big\{f:S^l\rightarrow M\Big\}\ \xrightarrow{\quad F\quad}\ \Big\{x:(B^{l-1},S^{l-2})\rightarrow (H^1(\T,M),M)\Big\}
\end{equation*}
which descends to a bijection between the sets of homotopy classes of continuous maps. Therefore we use a non zero-element $\mathfrak u\in \pi_l(M)$ to define 
$$\mathfrak U\ :=\ \Big\{\,u = (x,T):(B^{l-1},S^{l-2})\rightarrow (\Lambda_0,M_{T})\ \Big|\ F^{-1}(x)\in \mathfrak u\,\Big\}\, .$$

\begin{lemma}\label{lem:u}
The class $\mathfrak U$ enjoys the following two properties:
\begin{enumerate}
 \item It is invariant under $\hat \Phi^{k}$;
 \item Its elements intersect $\partial \mathcal V_{\delta}$: if $u\in\mathfrak U$, then $u(\zeta)\in \partial \mathcal V_{\delta}$ for some $\zeta \in B^{l-1}$.
\end{enumerate}
\end{lemma}
\begin{proof}
The former property follows from Proposition \ref{lem:com} and by the inclusion $M_{T}\subset \mathcal W'_{k}$ proved in Proposition \ref{lem:wdelta}. To show the latter property we suppose by contradiction that there is a $u\in\mathfrak U$ with $u(B^{l-1})\subset \mathcal V_{\delta}$ (observe that, by construction of the class $\mathfrak U$, every element $u\in \mathfrak U$ intersects $\mathcal V_\delta$). Since $M_{T}$ is a deformation retract of $\mathcal V_{\delta}$, there exists $u_1\in\mathfrak U$ such that $x_1(B^{l-1})\subset M$. Thus, $F^{-1}(x_1)$ represents the trivial class. This contradiction shows that such a $u$ cannot exist.
\end{proof}

We are now ready to define the minimax function $c^{\mathfrak u}:I\rightarrow (0,+\infty)$. We start by constructing a primitive of $\eta_k$ along any $u\in \mathfrak U$. Let
$N\in S^{l-2}=\partial B^{l-1}$ be the north pole and define $S_k(u):B^{l-1}\rightarrow \R$ as the unique function such that
\begin{equation}\label{eq:primonu}
d\big(S_k(u)\big)\ =\ u^*\eta_k\,,\quad\ \ S_k(u)(N)\ =\ S_k(u(N))\,.
\end{equation}
Notice that this definition makes sense since $B^{l-1}$ is simply connected and $u(N)$ lies in the domain of definition of the function $S_k$ introduced in \eqref{primitiveonvdelta}. 
\begin{remark}
When $l>2$ choosing a point in $S^{l-2}$ different from $N$ in the definition of $S_k(u)$ would lead to the same primitive since $S^{l-2}$ is connected. On the other hand, if $l=2$, choosing the south pole would yield a primitive which differs from the one considered here by a constant equal to the integral of $\sigma$ over $\mathfrak u$. 
\end{remark}
Using Equation \eqref{eq:var}, the function $S_k(u)$ can also be described in terms of the action variation along paths defined in \eqref{variazionelungou}. Indeed, for any $\zeta\in B^{l-1}$ consider the segment $z_\zeta:[0,1]\rightarrow B^{l-1}$ joining $N$ with $\zeta$ and let $u_\zeta:=u\circ z_\zeta$. There holds 
\begin{equation}\label{eq:primvar2}
S_k(u)(\zeta)\ =\ \Delta S_k\big(u_\zeta\big)(1)\ +\ S_k(u(N))\,,\quad\forall\, \zeta\in B^{l-1}\,.
\end{equation}
We now define the minimax function $c^{\mathfrak u}:I\rightarrow (0,+\infty)$ by
\begin{equation}
c^{\mathfrak u}(k)\ := \ \inf_{u\in \mathfrak U} \ \max_{\zeta\in B^{l-1}} \ S_k(u)(\zeta)\, ,
\label{eq:minimax}
\end{equation} 
In the next lemma we prove a crucial monotonicity property for the function $c^{\mathfrak u}$.

\begin{lemma}\label{lem:mon}
If $k_1,k_2\in I$, then for every $u\in\mathfrak U$ we have
\begin{equation}\label{eq:difper}
S_{k_2}(u)\ =\ S_{k_1}(u)\ +\ (k_2-k_1)\,T\,.
\end{equation}
In particular, if $k_1<k_2$, then
\begin{equation}\label{eq:c}
c^\mathfrak u(k_1)\ \leq\ c^\mathfrak u(k_2)\,.
\end{equation}
\end{lemma}

\begin{proof}
If $u\in\mathfrak U$, Equation \eqref{eq:skonm} implies that
\begin{equation*}
S_{k_2}(u(N))\ -\ S_{k_1}(u(N))\ =\ (k_2-k_1)\,T(N)
\end{equation*}
and, by the very definition of the action form in \eqref{etak}, we get
\begin{equation*}
d\big(S_{k_2}(u)-S_{k_1}(u)\big)\ =\ u^*\big(\eta_{k_2}-\eta_{k_1}\big)\ =\ (k_2-k_1)\,dT\,. 
\end{equation*}
Therefore, the function $S_{k_1}(u)+(k_2-k_1)T$ satisfies both conditions in \eqref{eq:primonu} with $k=k_2$. Since such conditions identify a unique function, Equation \eqref{eq:difper} follows.

If $k_1<k_2$ then \eqref{eq:difper} and the positivity of $T$ imply that $S_{k_1}(u)<S_{k_2}(u)$. Taking the inf-max of this inequality over $\mathfrak U$ yields \eqref{eq:c}.
\end{proof}

We use $c^{\mathfrak u}$ to construct a vanishing sequence for $\eta_k$ with periods bounded and bounded away from zero whenever $k$ is a Lipschitz-point for $c^{\mathfrak u}$. To exclude that the periods tend to zero we need the 
following lemma, which gives some information about the points which almost realize the maximum of $S_k(u)$.

\begin{lemma}\label{lem:notinwgeneral}
Fix $k\in I$ and $u\in\mathfrak U$. If $\zeta\in B^{l-1}$ is such that
\begin{equation}\label{ine:notinwgeneral}
S_k(u)(\zeta)\ \geq\ \max_{B^{l-1}}\ S_k(u)\ -\ \frac{\varepsilon}{2}\,,
\end{equation}
then $u(\zeta)\notin \mathcal W_{k}$.
\end{lemma}

\begin{proof}
Suppose by contradiction that $u(\zeta)\in \mathcal W_{k}$. By Lemma \ref{lem:u} there exists $\zeta_1\in B^{l-1}$ such that $u(\zeta_1)\in\partial \mathcal V_{\delta}$.
Denote with $z_\zeta^{\zeta_1}:[0,1]\rightarrow B^{l-1}$ the segment joining $\zeta$ with $\zeta_1$ and let $u_\zeta^{\zeta_1}:=u\circ z_\zeta^{\zeta_1}$. Without loss of generality we may suppose $u_\zeta^{\zeta_1}([0,1))\subseteq \mathcal V_\delta$. Applying twice \eqref{eq:var} we get
\begin{align*}
S_k(u)(\zeta_1)\ -\ S_k(u)(\zeta)\ =\ \Delta S_k\big(u_\zeta^{\zeta_1}\big)(1)\ =\ S_k(u(\zeta_1))\ -\ S_k(u(\zeta))\ >\ \frac{\varepsilon}{2}\,,
\end{align*}
where for the inequality we used Lemma \ref{lem:low} and the fact that $u(\zeta)\in \mathcal W_{k}$. This readily implies that
\begin{equation*}
\max_{B^{l-1}}\ S_k(u)\ \geq\ S_k(u)(\zeta_1)\ >\ S_k(u)(\zeta)\ +\ \frac{\varepsilon}{2}\,,
\end{equation*}
in contradiction with the assumption.
\end{proof}

\begin{proposition}
Let $c^{\mathfrak u}:I\rightarrow (0,+\infty)$ as defined in \eqref{eq:minimax}. If $k\in I$ is a Lipschitz point for $c^{\mathfrak u}$, then there exists a critical sequence $(x_h,T_h)\in \Lambda_0$ with periods
bounded and bounded away from zero.
\label{Struwe}
\end{proposition}

\begin{proof} 
By assumption there is $A>0$ such that for every $k'\geq k$ close to $k$ 
\begin{equation}
c^{\mathfrak u}(k')\ -\ c^{\mathfrak u}(k)\ \leq\ A\, (k'-k)\,.  
\label{stimalipschitz}
\end{equation}
Consider a sequence $k_m\downarrow k$ and let $\lambda_m:=k_m-k\downarrow 0$. Clearly we may suppose \eqref{stimalipschitz} to hold for every $m\in\N$. Take a corresponding $u_m=(x_m,T_m)\in\mathfrak U$ with
\begin{equation*}
\max_{\zeta\in B^{l-1}}\ S_{k_m}(u_{m})(\zeta)\ <\  c^{\mathfrak u}(k_m)\ +\ \lambda_m\, ,
\end{equation*}
Since $S_k(u_m)<S_{k_m}(u_m)$ by Lemma \ref{lem:mon}, we have that for all $\zeta\in B^{l-1}$
\begin{equation}
S_k(u_{m})(\zeta) \ \leq \  S_{k_m}(u_{m})(\zeta) \ <\ c^{\mathfrak u}(k_m)\ +\ \lambda_m \ \leq \ \ c^{\mathfrak u} (k)\ +\ (A+1)\, \lambda_m\, .
\label{tm1}
\end{equation}
If $\zeta\in B^{l-1}$ satisfies $S_k(u_m)(\zeta)> c^{\mathfrak u}(k)-\lambda_m$, we can use \eqref{eq:difper} to bound the period 
\begin{equation}
T_{m}(\zeta) =  \frac{S_{k_m}(u_{m})(\zeta)-S_{k}(u_{m})(\zeta)}{\lambda_m}
\leq\frac{c^{\mathfrak u}(k_m)+\lambda_m-\big (c^{\mathfrak u}(k)-\lambda_m\big )}{\lambda_m}\leq A+2\,.\label{tm}
\end{equation}

For every $\zeta\in B^{l-1}$ let $v^\zeta_m:[0,1]\rightarrow\Lambda_0$ be the flow line given by $v_m^\zeta(r):=\hat\Phi^k_r(u_m(\zeta))$. For every $r\in[0,1]$ define the element $u_m^r\in\mathfrak U$ by
\begin{equation*}
u_{m}^r(\zeta):=\ v_m^\zeta(r)\ =\ \hat \Phi^{k}_r(u_{m}(\zeta))\, , \quad \forall \ \zeta\in B^{l-1}\, .
\end{equation*}
By Equation \eqref{eq:primvar2} and \eqref{eq:dec} we can write
\begin{equation}\label{eq:var2}
S_k(u_m^r)(\zeta) \ =\ S_k(u_m)(\zeta)\ +\ \Delta S_k(v_m^\zeta)(r)\,,\quad\forall\,r\in[0,1]\,.
\end{equation}
By definition of action variation and the fact that $v^\zeta_m $ is a flow line we see that 
\[r\ \longmapsto \ S_{k}(u^r_{m})(\zeta)\]
is a non-increasing function $\forall\,\zeta\in B^{l-1}$. Combining this fact with \eqref{tm1} we get that 
\begin{equation}\label{eq:max}
\max_{\zeta\in B^{l-1}}\ S_{k}(u^r_{m})(\zeta)\ <\ c^{\mathfrak u}(k)\ +\ (A+1)\,\lambda_m\,, \quad \forall\, r\in[0,1]
\end{equation}

\noindent and the following dichotomy holds: either,
\begin{align*}
\mbox{\itshape (a)}\ \ S_{k}(u_{m}^1)(\zeta)\ &\leq\  c^{\mathfrak u}(k)\,-\,\lambda_m\,,\ \ \mbox{or}\\
\mbox{\itshape (b)}\ \ S_{k}(u_{m}^r)(\zeta)\ &\in\ \Big(c^{\mathfrak u}(k)\,-\,\lambda_m,\, c^{\mathfrak u}(k)\,+\,(A+1)\lambda_m\Big)\,,\ \ \forall\,r\in[0,1]\,.
\end{align*}

\noindent Suppose $\zeta$ satisfies the second alternative and let us draw some consequences of this case. First, by \eqref{eq:var2} we find that
\begin{equation}\label{eq:var3}
\Delta S_k(v_m^\zeta)(r)\ >\ -\,(A+2)\,\lambda_m\,,\quad\forall\,r\in[0,1]\,. 
\end{equation}
Second, Equation \eqref{eq:max} implies that
\begin{equation*}
S_{k}(u_{m}^r)(\zeta)\ >\ c^{\mathfrak u}(k)\ -\ \lambda_m\ >\ \max_{B^{l-1}}\ S_{k}(u_{m}^r)\ -\ (A+2)\, \lambda_m\,.
\end{equation*}

\noindent By Lemma \ref{lem:notinwgeneral}, $u^r_{m}(\zeta)\notin \mathcal W_\delta$, provided that $(A+2) \lambda_m\leq \varepsilon/2$, which is true for $m$ big enough. This shows that $v^\zeta_m$ is a genuine flow line of the untruncated semi-flow $\Phi^k$. Finally, we use \eqref{tm} and a combination of \eqref{eq:ac-per} and \eqref{eq:var3} to get 
\[T^r_{m}(\zeta)\ \leq\ T_{m}(\zeta)\,+\,\big|T^r_{m}(\zeta)-T_{m}(\zeta)\big |\ \leq\ A+2\ +\ \sqrt{r(A+2)\lambda_m}\ <\ A+3\,,\]

\noindent where the last inequality is true for $m$ sufficiently large.

After this preparation, we claim that there exists a critical sequence $(x_h,T_h)$ contained in $\{T<A+3\}\setminus \mathcal W_\delta$. 
This assertion readily implies the proposition.

To prove the claim we argue by contradiction and suppose that such a critical sequence does not exist. Then we can find $\varepsilon_*>0$ such that on $\{T<A+3\}\setminus \mathcal W_\delta$
\[|\eta_k|\ \geq\ |X_k| \ \geq \ \sqrt{\varepsilon_*}\, .\]

\noindent If $\zeta\in B^{l-1}$ satisfies the alternative \textit{(b)} above, by \eqref{eq:var3} and the fact that $v_m^\zeta$ is a flow line of $\Phi^k$ contained in $\{T<A+3\}\setminus \mathcal W_\delta$, we have
\begin{equation*}
-\,(A+2)\,\lambda_m\ <\ \Delta S_k(v_{m}^\zeta)(1)\ =\ -\int_0^1\big | \eta_k(X_k)\big |^2\, dr\ \leq \ - \varepsilon_*\,,
\end{equation*}
Hence, the set of $\zeta\in B^{l-1}$ satisfying \textit{(b)} is empty as soon as $(A+2)\lambda_m<\varepsilon_*$.

Therefore, for $m$ big enough, all the $\zeta\in B^{l-1}$ satisfy \textit{(a)}. Since $u^1_m$ belongs to $\mathfrak U$, this contradicts the definition of $c^{\mathfrak u}(k)$ and finishes the proof.
\end{proof}

\begin{proof}[Proof of Theorem \ref{theoremb}] Combine the previous proposition with Theorem \ref{theorem:ps} and the fact that a monotone function is Lipschitz-continuous almost everywhere.
\end{proof}

In the particular case $M=S^2$, Theorem \ref{theoremb} together with the stability property for energy levels implies the following result about the existence of closed magnetic geodesics. Both statements below are actually 
already known (see, for instance, \cite[Theorem 3.13]{Gin96} and \cite[Assertion 3]{Gin87}, respectively).

\begin{corollary}\label{corollariosfera}
Let $(S^2,g)$ be a Riemannian 2-sphere and $\sigma\in \Omega^2(S^2)$. Then:
\begin{enumerate}
\item If $k$ is large enough, there is a closed magnetic geodesic with energy $k$.
\item Suppose $\sigma$ is symplectic. If $k>0$ is small enough, there exists a closed magnetic geodesic with energy $k$.
\end{enumerate}
\end{corollary}


\section{Contractible orbits for weakly exact magnetic forms}\label{sec:we}

In this section we give a sketch of proof of Theorem \ref{theorema}. Thus, let $\sigma$ be a weakly exact form. Under this assumption $\eta_k$ is exact on $\Lambda_0$ and a primitive is given by 
\begin{equation}
S_k(x,T)\, := \ S_k^L(x,T) \ + \ \int_{B^2} D_x^*\,\sigma\, ,
\label{secondadefdisn}
\end{equation}
where $D_x:B^2\rightarrow M$ is any capping disc for $x$. Notice that $S_k$ is well-defined since 
\[\int_{B^2} D_x^*\,\sigma\]
does not depend on the choice of the disc $D_x$, because $\sigma|_{\pi_2(M)}=0$. Moreover, the function $S_k$ defined in \eqref{secondadefdisn} extends the primitive of $\eta_k$ on short loops introduced in \eqref{primitiveonvdelta}. Finally, $S_k$ enjoys the following crucial property (see \cite[Lemma 2.1]{Mer10}):
\begin{equation*}
\inf_{\Lambda_0}\ S_k\ =\ -\infty\,, \ \ \ \ \forall \ k < c(H,\sigma)\, .
\end{equation*}

Our goal is to show that for almost every $k\in (e_0(H),c(H,\sigma))$ there exists a contractible periodic orbit of $\Phi^{H,\sigma}$ with energy $k$. We do this using a minimax method analogous to the one exploited in Section \ref{sec:gen}, but with a different minimax class as we now explain. Fix an open interval $I$ with compact closure in $(e_0(H),c(H,\sigma))$ and let $\delta=\delta(I)$, $\varepsilon=\varepsilon(I)$ be as in Lemma \ref{lem:low} and $T=T(I)$ as in Proposition \ref{lem:wdelta}. Pick a constant loop $(x_0,T)\in M_T$ 
and define 
\[\mathfrak U\ := \ \Big \{\, u:[0,1]\rightarrow \Lambda_0 \ \Big |\ u(0) = (x_0,T)\, , \ S_k(u(1))<0\,,\ \forall\, k\in I\,\Big \}\,.\]

Lemma \ref{lem:u} readily follows also in this setting. Indeed, since $M_{T}\subset \mathcal W'_{k}$ and $S_k$ does not increase along the flow lines of $\hat \Phi^{k}$ we conclude that the class $\mathfrak U$ is invariant under the semi-flow $\hat \Phi^{k}$. Furthermore, since $S_k$ is non-negative on $\mathcal V_{\delta}$ by Lemma \ref{lem:low}, every element $u\in \mathfrak U$ has to intersect $\partial \mathcal V_{\delta}$. 

Let $c:I\rightarrow (0,+\infty)$ denote the minimax function defined by 
\[c(k)\, := \ \inf_{u\in \mathfrak U} \max_{s\in [0,1]} \ S_k(u(s))\,.\]
As the family of functions $k\mapsto S_k$ is increasing, the function $c$ is non-decreasing. Thus, Lemma \ref{lem:mon} still holds. Finally, since all $u\in\mathfrak U$ intersect $\partial \mathcal V_{\delta}$, we know that 
\[\max_{s\in [0,1]}\  S_k(u(s)) \ \geq\ \varepsilon\, .\]

\noindent This readily implies that Lemma \ref{lem:notinwgeneral} remains true in this case and that $c(k)\geq \varepsilon$. Now, a minor adaptation of the proof of Proposition \ref{Struwe} shows that for every $k\in  I$ at which the function $c$ is Lipschitz-continuous there is a vanishing sequence for $\eta_k$ with periods bounded and bounded away from zero. Theorem \ref{theorema} then follows using Theorem \ref{theorem:ps} and the fact that $c$ is almost everywhere Lipschitz-continuous.

As a sample application of stability, one can reprove a weak version of a theorem in \cite{Gin87}, without the multiplicity results. 

\begin{corollary}
Suppose $(M,g)$ is a Riemannian surface with positive genus and $\sigma\in \Omega^2(M)$ is a symplectic 2-form. If $k>0$ is small enough, there exists a contractible closed magnetic geodesic with energy $k$.
\label{corollariosuperfici}
\end{corollary}


\section{Contractible periodic orbits on displaceable energy levels}\label{sec:dis}

In this section we prove Theorem \ref{theoremc} and Corollary \ref{corollaryc}. Thus, let $M$ be a closed manifold and $\sigma$ any closed $2$-form on it. Consider a smooth function $f:M\rightarrow\R$. The
key observation is that the Hamiltonian flow of $f\circ\pi:T^*M\rightarrow\R$ is a fibrewise translation by $-df$ (in particular it does not depend on $\sigma$).

\begin{lemma}
For every $t\in\R$ and $(q,p)\in T^*M$, there holds
\[\Phi^{f\circ\pi,\sigma}_t(q,p)\ =\ (q,p-t\, d_qf)\, .\]
\end{lemma}

\begin{proof}
Let us write down the coordinate expression of the Hamilton equations for an arbitrary function $K:T^*M\rightarrow\R$:
\begin{equation}\label{eq:ham}
\left\{
\begin{aligned}
\dot q\ &=\ \frac{\partial K}{\partial p}\,,\\
\dot p\ &=\ -\frac{\partial K}{\partial q}\ -\ \sigma\left(\frac{\partial
K}{\partial p},\cdot\right).
\end{aligned}
\right.
\end{equation}
When $K=f\circ\pi$ they reduce to $\dot q=0$, $\dot{p}=-df$ and the lemma follows.
\end{proof}

\begin{proof}[Proof of Theorem \ref{theoremc}]
We have to show that if $k<e_0(H)$, then $\{H\leq k\}$ is displaceable in $(T^*M,\omega_\sigma)$. By the very definition of $e_0(H)$, the set
$$M\setminus \pi(\{H\leq k\})$$
is open and non-empty for every $k<e_0(H)$. Hence, there is a function $f:M\rightarrow \R$ whose critical points are all disjoint from  $\pi(\{H\leq k\})$ by \cite[Lemma 8.1]{Con06}. In particular, if $g$ is any Riemannian metric on $M$, 
then there exists $\varepsilon>0$ such that
\begin{equation*}
\inf_{q\in\pi(\{H\leq k\})}\ |d_qf|_q\ \geq\ \varepsilon\,.
\end{equation*}

On the other hand, since the sublevel sets of $H$ are compact, there also exists $B>0$ such that 
\begin{equation*}
\{H\leq k\}\ \subset\ \Big \{(q,p)\in T^*M \ \Big |\ |p|_q\leq B\Big \}\, .
\end{equation*}
If $T>B/\varepsilon$ and $(q,p)\in \{H\leq k\}$, the superlinearity of $H$ implies that
\begin{align}
H(q,p-T\, d_qf) \ &\geq\ H_0\, |p\,-\,T\, d_qf|_q\ -\ H_1 \nonumber \\ 
                                &\geq\ H_0\big(T\, |d_qf|_q\,-\,|p|_q\big)\ -\ H_1\nonumber \\ 
                                &\geq\ H_0\varepsilon T\ -\ H_0B\ -\ H_1 \label{ggg}
\end{align}
for some positive constants $H_0,H_1$. When $T$ is large enough, the quantity in \eqref{ggg} is bigger than $k$ and therefore
\begin{equation*}
\Phi^{f\circ\pi,\sigma}_{T}\ =\ \Phi^{Tf\circ\pi,\sigma}_1
\end{equation*}
is a Hamiltonian diffeomorphism which displaces the set $\{H\leq k\}$. Consider now a compactly supported function $\chi:T^*M\rightarrow\R$ which is equal to $1$ on a neighbourhood of
\begin{equation}
\Big \{\Phi^{f\circ\pi,\sigma}_t(q,p)\ \Big |\ (q,p)\in \{H\leq k\}\,,\ t\in[0,T]\Big \}\,.
\end{equation}

\noindent Then, $\chi\cdot (Tf\circ\pi)$ is a compactly supported Hamiltonian whose time-$1$ flow map displaces $\{H\leq k\}$ and this completes the proof.
\end{proof}

\begin{remark}
Other classes of Hamiltonians can be used to displace the sublevels of $H$. In the case $\sigma=0$, Contreras took a Hamiltonian of the kind $F(q,p)=p(\nabla f_q)$, where $\nabla f$ is the gradient vector field of a smooth function
$f:M\rightarrow\R$, whose critical points are disjoint from $\pi(\{H\leq k\})$ (see \cite[Proposition 8.2]{Con06}). In his argument displaceability follows since the Hamiltonian flow of $F$ lifts the gradient
flow of $f$ and $\pi(\{H\leq k\})$ is displaced in $M$ by such flow. This choice is equally good in the general case we consider here, as the Hamilton equations \eqref{eq:ham} for the function $F$ still imply $\dot q=\nabla f$.
\end{remark}

\begin{proof}[Proof of Corollary \ref{corollaryc}]
First, we observe that $(T^*M,\omega_\sigma)$ is tame and strongly semi-positive. Tameness was proved in \cite{CGK04} (there the term ``geometrically bounded" is used). Strong semi-positivity is due to the fact
that the first Chern class of $\omega_\sigma$ vanishes since the cotangent fibres form a Lagrangian foliation of $(T^*M,\omega_\sigma)$. Thus, the hypotheses of Theorem 1.2 of \cite{Sch06} are satisfied and we can apply
Corollary 3.2 of \cite{Sch06} to $H^{-1}(k)$, for $k\in[\min H,e_0(H))$, by Theorem \ref{theoremc}. This implies the existence of contractible periodic orbits on almost every energy level in the energy range $[\min H, e_0(H))$, 
thus completing the proof.
\end{proof}


\bibliographystyle{amsalpha}
\bibliography{wes}

\end{document}